\documentclass[11pt]{article}

\usepackage{mathrsfs}
\usepackage{makecell}
\usepackage{algorithmic}
\usepackage{algorithm}
\usepackage{amscd}
\usepackage{amsmath}
\usepackage{amssymb}
\usepackage{amstext}
\usepackage{amsthm}
\usepackage{bbold}
\usepackage{bm}
\usepackage{booktabs}
\usepackage{color}
\usepackage{easybmat}
\usepackage{framed}
\usepackage[dvips,letterpaper,margin=1in]{geometry}
\usepackage{graphicx}
\usepackage{hyperref}
\usepackage[noabbrev,capitalize]{cleveref}
\usepackage{mathtools}
\usepackage{colonequals}
\usepackage{longtable}
\usepackage{rotating}
\usepackage{setspace}
\usepackage{tabu}
\usepackage{verbatim}

\newtheorem{theorem}{Theorem}[section]
\newtheorem{lemma}[theorem]{Lemma}

\newtheorem{definition}[theorem]{Definition}
\newtheorem{problem}[theorem]{Problem}

\newtheorem{proposition}[theorem]{Proposition}

\newtheorem{conjecture}[theorem]{Conjecture}
\newtheorem{remark}[theorem]{Remark}

\crefname{theorem}{Theorem}{Theorems}
\crefname{proposition}{Proposition}{Propositions}
\crefname{lemma}{Lemma}{Lemmas}

\theoremstyle{plain} % just in case the style had changed
\newcommand{\thistheoremname}{}
\newtheorem*{genericthm}{\thistheoremname}

\DeclareSymbolFont{bbold}{U}{bbold}{m}{n}
\DeclareSymbolFontAlphabet{\mathbbold}{bbold}

\newcommand{\ra}{\rangle}
\newcommand{\la}{\langle}

\makeatletter
\def\moverlay{\mathpalette\mov@rlay}
\def\mov@rlay#1#2{\leavevmode\vtop{%
   \baselineskip\z@skip \lineskiplimit-\maxdimen
   \ialign{\hfil$\m@th#1##$\hfil\cr#2\crcr}}}
\newcommand{\charfusion}[3][\mathord]{
    #1{\ifx#1\mathop\vphantom{#2}\fi
        \mathpalette\mov@rlay{#2\cr#3}
      }
    \ifx#1\mathop\expandafter\displaylimits\fi}
\makeatother

\newcommand{\RR}{\mathbb{R}}
\newcommand{\QQ}{\mathbb{Q}}
\newcommand{\NN}{\mathbb{N}}

\newcommand{\PP}{\mathbb{P}}
\newcommand{\EE}{\mathbb{E}}

\newcommand{\tp}{\textit{\texttt{p}}}
\newcommand{\tq}{\textit{\texttt{q}}}

\newcommand{\one}{\bm{1}}

\newcommand{\sN}{\mathcal{N}}

\newcommand{\sX}{\mathcal{X}}

\newcommand{\Ex}{\mathop{\mathbb{E}}}
\newcommand{\Var}{\mathop{\text{Var}}}

\newcommand{\NP}{\mathsf{NP}}

\newcommand{\polylog}{\mathsf{polylog}}

\newcommand{\bP}{\bm P}

\newcommand{\RIP}{\mathsf{RIP}}

\newif\ifnotes
\notestrue

\definecolor{AfonsoBlue}{RGB}{30,65,123}

\makeatletter
\renewcommand*{\@fnsymbol}[1]{\ensuremath{\ifcase#1\or *\or \ddagger\or
    \mathsection\or \mathparagraph\or \|\or **\or \dagger\dagger
    \or \ddagger\ddagger \else\@ctrerr\fi}}
\makeatother

\title{The Average-Case Time Complexity of Certifying the Restricted Isometry Property}

\usepackage{authblk}
\author[1]{Yunzi Ding\thanks{Email: \textit{yding@nyu.edu}. Partially supported by NSF grant DMS-1712730.}}
\author[1]{Dmitriy Kunisky\thanks{Email: \textit{kunisky@cims.nyu.edu}. Partially supported by NSF grants DMS-1712730 and DMS-1719545.}}
\author[1]{Alexander S.\ Wein\thanks{Email: \textit{awein@cims.nyu.edu}. Partially supported by NSF grant DMS-1712730 and by the Simons Collaboration on Algorithms and Geometry.}}
\author[2]{Afonso S.\ Bandeira\thanks{Email: \textit{bandeira@math.ethz.ch}. Part of this work was done while ASB was with the Department of Mathematics at the Courant Institute of Mathematical Sciences, and the Center for Data Science, at New York University; and partially supported by NSF grants DMS-1712730 and DMS-1719545, and by a grant from the Sloan Foundation.}}

\affil[1]{Department of Mathematics, Courant Institute of Mathematical Sciences, New York University, USA}
\affil[2]{Department of Mathematics, ETH Zurich, Switzerland}

\date{}
\begin{document}

\maketitle

\begin{abstract}
In compressed sensing, the restricted isometry property (RIP) on $M \times N$ sensing matrices (where $M < N$) guarantees efficient reconstruction of sparse vectors. A matrix has the $(s,\delta)$-$\RIP$ property if behaves as a $\delta$-approximate isometry on $s$-sparse vectors. It is well known that an $M\times N$ matrix with i.i.d.\ $\sN(0,1/M)$ entries is $(s,\delta)$-$\RIP$ with high probability as long as $s\lesssim \delta^2 M/\log N$. On the other hand, most prior works aiming to deterministically construct $(s,\delta)$-$\RIP$ matrices have failed when $s \gg \sqrt{M}$. An alternative way to find an RIP matrix could be to draw a random gaussian matrix and certify that it is indeed RIP. However, there is evidence that this certification task is computationally hard when $s \gg \sqrt{M}$, both in the worst case and the average case.

In this paper, we investigate the exact average-case time complexity of certifying the RIP property for $M\times N$ matrices with i.i.d.\ $\sN(0,1/M)$ entries, in the ``possible but hard'' regime $\sqrt{M} \ll s\lesssim M/\log N$. Based on analysis of the low-degree likelihood ratio, we give rigorous evidence that subexponential runtime $N^{\tilde\Omega(s^2/M)}$ is required, demonstrating a smooth tradeoff between the maximum tolerated sparsity and the required computational power. This lower bound is essentially tight, matching the runtime of an existing algorithm due to Koiran and Zouzias~\cite{KZ-rip}. Our hardness result allows $\delta$ to take any constant value in $(0,1)$, which captures the relevant regime for compressed sensing. This improves upon the existing average-case hardness result of Wang, Berthet, and Plan \cite{WBP-avg-hard}, which is limited to $\delta = o(1)$.\\

\end{abstract}

\newpage

\newpage

\section{Introduction}
\subsection{Restricted Isometry Property}
For measuring and reconstructing high-dimensional sparse signals, the compressed sensing technique introduced by Cand\`{e}s and Tao \cite{CT-decoding} and Donoho \cite{donoho-comp-sensing} has demonstrated state-of-the-art efficiency and effectiveness in theory and practice. A central property on the sensing matrix, known as the \emph{restricted isometry property (RIP)} \cite{candes-rip}, requires that the matrix approximately preserves the norm of sparse vectors.
\begin{definition}
\label{def-rip}
A matrix $A\in\RR^{M\times N}$ is said to satisfy the $(s,\delta)$-$\RIP$ if
\[(1-\delta)\|x\|^2 \le \|Ax\|^2 \le (1+\delta)\|x\|^2 \]
 for any $x\in \RR^N$ with $\|x\|_0 \le s$. Here $\|\cdot\|$ denotes the vector $\ell^2$ norm, and $\|\cdot\|_0$ denotes the number of nonzero entries of a vector.
\end{definition}

\noindent In the context of compressed sensing, in pursuit of reducing the dimension, we are interested in the case $M < N$. An $(s,\delta)$-$\RIP$ sensing matrix with \mbox{$\delta < \sqrt{2}-1$} allows for efficient reconstruction of an $s/2$-sparse $N$-dimensional signal from $M$ linear measurements in the compressed sensing framework \cite{candes-rip,CZ-recovery,CZ-recovery2,foucart-cs}. From a practical point of view, given desired parameters $M,N,s,\delta$, one would like to construct an $(s,\delta)$-$\RIP$ matrix suitable for use as a sensing matrix. It appears to be very difficult to deterministically construct RIP matrices for $s \gg \sqrt{M}$, a phenomenon known as the ``square bottleneck'' \cite{mixon-explicit,BFMM-derandomizing,BMM-conditional,gamarnik-ramsey-hard}; in fact, the only known success is due to \cite{bourgain}, which constructed $(s,\delta)$-$\RIP$ matrices for
%$s\sim M^{\frac{1}{2}+\epsilon}$ where $0 < \epsilon \ll 1$.
$s \ge M^{1/2 + \epsilon}$ for a small constant $\epsilon > 0$. Readers may refer to \cite{BFMW-deterministic} for more details.

Meanwhile, randomized algorithms have seen much success in breaking the ``square bottleneck'' \cite{CT-decoding,baraniuk-proof,foucart-cs}. For example, by simply sampling an $M\times N$ matrix with i.i.d.\ $\sN(0,1/M)$ entries, one obtains a $(s,\delta)$-$\RIP$ matrix with high probability (i.e., probability $1-o(1)$), so long as $s\lesssim \delta^2 M/\log N$. While such randomized algorithms generate RIP matrices with more desireable parameters $s,\delta$ than the known deterministic constructions, they suffer from a potential drawback: it is not guaranteed that the output is \emph{always} $(s,\delta)$-$\RIP$ as desired. This motivates the following task known as RIP certification, in which we discard non-RIP samples and confidently keep a sample only when it is indeed RIP.

\begin{problem}[RIP certification]
\label{rip-cert}
Given a matrix $A\in\RR^{M\times N}$, a positive integer $s$, and $\delta\in (0,1)$, output either ``yes'' or ``no'' according to the following rules. If $A$ is not $(s,\delta)$-$\RIP$, the output must always be ``no''. If $A$ has i.i.d.\ $\mathcal{N}(0,1/M)$ entries, the output must be ``yes" with high probability.
\end{problem}

\noindent Note that this allows false negative errors but not false positive errors, so that such a certifying procedure allows us to reliably obtain an RIP matrix by sampling random matrices until the certifier outputs ``yes''.

The \textit{worst-case} problem of deciding (with certainty) whether or not a given matrix is $(s,\delta)$-$\RIP$ is $\NP$-hard \cite{TP-complexity,BDMS-hard}, even if the input is guaranteed to either be RIP or far from RIP \cite{weed-approx-hard} (see also~\cite{KZ-rip,NW-sse} which require stronger assumptions than $\mathsf{P} \ne \NP$). All known algorithms for this worst-case task~\cite{devore-deterministic,applebaum-chirp,FMT-steiner-etf,BFMW-deterministic} require time $N^{\tilde\Omega(s)}$, which is the time required to enumerate all possible support sets $S \subseteq [N]$ of cardinality $s$.

For the \emph{average-case} RIP certification problem (Problem~\ref{rip-cert}), it is shown in \cite{WBP-avg-hard} that thresholding $\|A^\top A-I_N\|_{\infty}$ gives a polynomial-time certifier of $(s,\delta)$-$\RIP$ for i.i.d.\ sub-gaussian matrices in the regime $s\lesssim \delta\sqrt{M/\log N}$. Note, however, that this does not help surpass the ``square bottleneck''. In fact, the same work \cite{WBP-avg-hard} shows that for any $\epsilon > 0$, when $s \ge (\delta^2 M/\log N)^{1/2 + \epsilon}$, no polynomial-time certifier exists for the average-case task, conditional on an assumption about detecting dense subgraphs, which is a weaker assumption than the \textit{planted clique hypothesis} (i.e., their assumption is implied by the planted clique hypothesis). However, this result only shows hardness in the regime $\delta = o(1)$.

\subsection{Our Contributions}

In this paper, we further investigate the average-case hardness of Problem~\ref{rip-cert}. For any fixed $\delta\in(0,1)$, in the ``possible but hard'' regime $\sqrt{M} \ll s \lesssim M/\log N$, we give evidence that $(s, \delta)$-$\RIP$ certification requires time $N^{\tilde\Omega(s^2/M)}$ with an analysis of the \textit{low-degree likelihood ratio} (see Section~\ref{low-deg}). Our lower bound is optimal, as it matches the runtime of the \textit{lazy algorithm} proposed in~\cite{KZ-rip}. Here and throughout, the notation $\tilde \Omega$, $\tilde{O}$, and $\tilde{\Theta}$ hide factors of $\log N$, while $\lesssim$ hides a constant factor.

The strategy for our lower bound is to give a reduction to Problem \ref{rip-cert} from a certain hypothesis testing problem in the negatively-spiked Wishart model. Our arguments will be based on two distributions over $\RR^{M\times N}$, namely the distribution $\QQ$ of i.i.d.\ Gaussian matrices and a distribution $\PP$ over matrices which have a sparse vector planted in their null-space. We start with a proof that, in the regime $\sqrt{M} \ll s \lesssim M/\log N$, $A\sim \QQ$ is RIP with high probability while $A\sim \PP$ is non-RIP with high probability. We then argue that it is computationally hard to distinguish $\PP$ from $\QQ$, assuming the \emph{low-degree conjecture} (see Section~\ref{low-deg}); in other words, we prove that the class of \emph{low-degree polynomial algorithms} requires time $N^{\tilde\Omega(s^2/M)}$ to distinguish $\PP$ from $\QQ$. From this we infer the hardness of RIP certification.

For the matching upper bound, let us give a brief overview of the \textit{lazy algorithm} in~\cite{KZ-rip}. The main step of the algorithm involves exhaustive search over subsets of $[N]$ with cardinality $r\approx s^2/M$, which are interpreted as the possible supports of $r$-sparse vectors in $\mathbb{R}^N$. As $s$ ranges from $\sqrt{M}$ up to $M/\log N$, the runtime of the algorithm ($N^{\tilde O(s^2/M)}$) smoothly interpolates between polynomial ($N^{O(1)}$) and exponential ($N^{\tilde{O}(M)}$). Our lower bound suggests that this is in fact the \emph{optimal} tradeoff between runtime and sparsity.

Our main contribution in this paper is a precise understanding of the computational power needed for the average-case certification of $(s,\delta)$-$\RIP$, namely $N^{\tilde\Theta(s^2/M)}$. In contrast, the previous average-case hardness result of \cite{WBP-avg-hard} only suggests that at least $N^{\log(N)}$ time is required, since (like planted clique) the dense subgraph problem that they give a reduction from can be solved in time $N^{\log(N)}$. Another strength of our lower bound is that it applies for \emph{any} fixed $\delta \in (0,1)$, whereas the hardness result of~\cite{WBP-avg-hard} is restricted to $\delta = o(1)$. In other words, we are showing that even an \emph{easier} certification problem is hard. Since any fixed $\delta < \sqrt{2} - 1$ is sufficient for compressed sensing applications, our result is the first average-case lower bound to capture the entire regime of relevant $\delta$ values.

\subsection{Spiked Wishart Model and Hardness of RIP Certification}\label{sec:wishart}

Definition \ref{def-rip} tells us that a matrix is not $(s,\delta)$-$\RIP$ for any $\delta \in (0,1)$ if there exists an $s$-sparse vector in its kernel. The computational hardness of RIP certification in this paper is based on the following fact:\\

\noindent\textbf{Any certifier for $(s,\delta)$-$\RIP$ (a solution to Problem \ref{rip-cert}) can be used to distinguish a random matrix from any matrix that has an $s$-sparse vector in its kernel.}\\

\noindent Note that the operation on $A\in \RR^{M\times N}$ that projects each row of $A$ onto the subspace $\{x\}^\perp$ for some $s$-sparse vector $x\in\RR^N$ adds $x$ to the kernel of the resulting matrix, and hence deprives the matrix $A$ of its $(s,\delta)$-$\RIP$ property. Therefore, a certifier should be able to tell whether an $(s,\delta)$-$\RIP$ matrix $A$ has undergone this operation. We will make use of the classical spiked Wishart model, which captures the property of the row-wise projection.

\begin{definition}[Spiked Wishart model] 
\label{def-wish}
Let $\sX = (\sX_N)$ be a distribution over $\RR^N$, and let $\beta\in [-1,+\infty)$. We define two distributions over $\RR^{M\times N}$, where $A\in\RR^{M\times N}$ is taken as follows:
\begin{itemize}
    \item Under $\QQ = \QQ_{N,M}$, draw each row $u_i^\top$ ($i = 1,2,\dots,M$) of $A$ i.i.d.\ from $\sN(0,I)$.
    \item Under $\PP = \PP_{N,M,\beta,\sX}$, draw $x\sim \sX$. If $\beta \|x\|^2 \ge -1$, then draw each row $u_i^\top$ ($i = 1,2,\dots,M$) of $A$ i.i.d.\ from $\sN(0,I+\beta xx^\top)$; otherwise, draw each row $u_i^\top$ ($i = 1,2,\dots,M$) of $A$ i.i.d.\ from $\sN(0,I)$.
\end{itemize}
We call $\PP$ the planted model and $\QQ$ the null model. We define the spiked Wishart model as the two distributions taken together: $(\PP,\QQ) = \mathsf{Wishart}(N,M,\beta,\sX)$.
\end{definition}

\noindent
We will consider spike priors $\sX$ normalized so that $\|x\| \approx 1$, ensuring that $\beta \|x\|^2 \ge -1$ is a high-probability event under $\PP$. The condition $\beta \|x\|^2 \ge -1$ in the definition of $\PP$ ensures that the covariance matrix $I+\beta xx^\top$ is positive semidefinite.
Our case of interest will be $\beta < 0$, in which case $\PP$ can be viewed as the row-wise partial projection of $\QQ$ onto $\{x\}^\perp$, where $x$ is taken from the prior distribution $\sX$. In fact, for any $u \sim \sN(0,I+\beta xx^\top)$ where $\|x\| = 1$, we have 
\[\EE\la u,x\ra^2 = 1+\beta < 1 = \EE \la u,y\ra^2,\quad \text{for any unit-norm } y\perp x \]
which tells us that the rows of $A$ under $\PP$ have low correlation to the subspace $\{x\}$ but high correlation to the subspace $\{x\}^\perp$.

Since we are interested in the signal $x$ being sparse, we will take the spike prior $\sX$ to be the following sparse Rademacher prior.

\begin{definition}[Sparse Rademacher prior]
\label{sps-rad}
Given $\rho\in (0,1)$, the sparse Rademacher prior $\sX_N^\rho$ is the distribution over $\RR^N$ where for $x\sim \sX_N^\rho$ each entry $x_i$ is distributed independently as
\begin{equation*}
x_i = \left\{
\begin{aligned}
\frac{1}{\sqrt{\rho N}} \quad &\text{with probability }\frac{\rho}{2},\\
-\frac{1}{\sqrt{\rho N}} \quad &\text{with probability }\frac{\rho}{2},\\
0 \quad &\text{with probability }1-\rho.\\
\end{aligned}
\right.
\end{equation*}
\end{definition}

\noindent
Note that $x\sim \sX$ has $\EE \|x\|^2 = 1$ and $\EE \|x\|_0 = \rho N$. In order for $x$ to be $s$-sparse with high probability, we will take $\rho = s/(2N)$.

We now introduce the logic for demonstrating the average-case hardness of certifying RIP for gaussian random matrices. The following proposition serves as an outline of the proof of the lower bound. The informal assertions $(1),(2)$ and $(3)$ are made rigorous in Section \ref{sec:main-results}, in Lemma \ref{P-notrip}, Theorem \ref{Q-rip} and Theorem \ref{ldlr-bdd} respectively. 

\begin{proposition}[Informal] 
\label{pca-rip}
Fix a constant $\delta = (0,1)$ and let $\epsilon = \frac{1-\delta}{2(1+\delta)}$. Consider an asymptotic regime where $N \to \infty$ and where $M = M(N)$ and $s = s(N)$ scale with $N$ such that $M \le N$, $M \to \infty$ and $s \to \infty$.
Consider the sequence of spiked Wishart models with sparse Rademacher prior $\{\mathsf{Wishart}(N,M,-(1-\epsilon),\sX^\rho_N)\}_{N\in \mathbb{Z}^+}$ where $\rho = s/(2N)$. The following three properties hold in the limit $N\rightarrow\infty$.
\begin{itemize}
    \item[(1)] For $A\sim \PP$, $\frac{1}{\sqrt{M}}A$ is not $(s,\delta)$-$\RIP$ with high probability.
    \item[(2)] For $A\sim \QQ$, $\frac{1}{\sqrt{M}}A$ is $(s,\delta)$-$\RIP$ with high probability, provided $s \lesssim M/\log N$.
    \item[(3)] Conditional on the low-degree conjecture (see Section~\ref{low-deg}), any algorithm to distinguish $\PP$ from $\QQ$ with error probability $o(1)$ requires time $N^{\tilde\Omega(s^2/M)}$.
\end{itemize}
\end{proposition}

\noindent
In light of assertions (1) and (2), any certifier for $(s,\delta)$-$\RIP$ (a solution to Problem~\ref{rip-cert}) can be used to distinguish $\PP$ from $\QQ$ with error probability $o(1)$. Thus (3) implies that, conditional on the low-degree conjecture, any algorithm to solve Problem~\ref{rip-cert} requires time $N^{\tilde\Omega(s^2/M)}$.

The problem of distinguishing two sequences of distributions falls into the setting of hypothesis testing. In the following Section \ref{low-deg}, we discuss a general method to predict computational hardness of such tasks.

\subsection{The Low-Degree Likelihood Ratio}
\label{low-deg}

The so-called \textit{low-degree method} for predicting the amount of computational power required for hypothesis testing tasks has seen fruitful development in the recent years, after its origination from studies on the sum-of-squares (SoS) hierarchy \cite{BHK-planted-clique,HS-bayesian,HKP-sos,hopkins-thesis}. This method has proven successful in understanding many classical statistical tasks, including community detection \cite{HS-bayesian,hopkins-thesis}, planted clique \cite{BHK-planted-clique,hopkins-thesis}, PCA and sparse PCA in spiked matrix models \cite{BKW-sk,KWB-notes,DKWB-spspca}, and tensor PCA \cite{HKP-sos,hopkins-thesis,KWB-notes}. Here we give a brief overview of this method; the reader may find more details in \cite{HS-bayesian,hopkins-thesis} or in the survey article \cite{KWB-notes}.

This method applies to hypothesis testing problems in which we aim to distinguish two sequences of hypotheses $\{\PP_N\}$ and $\{\QQ_N\}$, where $\PP_N$ and $\QQ_N$ are probability distributions on $\Omega_N = \RR^{d(N)}$ with $d(N) = N^{O(1)}$. Usually $\QQ_N$ is referred to as the ``null'' distribution (which contains pure noise), and $\PP_N$ is referred to as the ``planted'' distribution (which contains a planted structure). We are interested in the problem of \emph{strongly distinguishing} $\PP$ from $\QQ$ in a computationally-efficient manner, that is, we are interested in the minimal runtime needed for an algorithm that takes as input a sample from either $\PP$ or $\QQ$ and correctly identifies which of the two distributions it was drawn from, with both type I and type II errors $o(1)$ as $N \to \infty$.

The central idea of the low-degree method is to study a restricted class of algorithms, namely degree-$D$ multivariate polynomials $f: \Omega_N \to \RR$, for a particular choice of $D = D(N)$ discussed later. We say that such a polynomial succeeds at distinguishing $\PP$ from $\QQ$ provided that
\begin{equation}\label{eq:f-success}
\Ex_{Y \sim \QQ_N}[f(Y)] = 0, ~ \Ex_{Y \sim \PP_N}[f(Y)] = 1, ~ \Var_{Y \sim \QQ_N}[f(Y)] = o(1), ~ \Var_{Y \sim \PP_N}[f(Y)] = o(1)
\end{equation}
as $N \to \infty$, in which case $\PP$ and $\QQ$ can be strongly distinguished by thresholding $f(Y)$. To prove failure of \emph{all} degree-$D$ polynomials, we will be interested in computing the quantity
\[ \|L_N^{\le D}\| := \max_{f \in \RR[Y]_{\le D}} \frac{\EE_{Y \sim \PP_N}[f(Y)]}{\sqrt{\EE_{Y \sim \QQ_N}[f(Y)^2]}} \]
where $\RR[Y]_{\le D}$ is the space of polynomials $f:\Omega_N \to \RR$ of degree (at most) $D$. Note that if $\|L_N^{\le D}\| = O(1)$ as $N \to \infty$ then no degree-$D$ polynomial can succeed in the sense of~\eqref{eq:f-success}. The notation $\|L_N^{\le D}\|$ stems from the fact that this quantity is the $L^2(\QQ_N)$-norm of the \emph{low-degree likelihood ratio} $L_N^{\le D}$, which is the orthogonal projection of the likelihood ratio $L_N = d\PP_N/d\QQ_N$ onto the subspace of degree-$D$ polynomials; see e.g.\ \cite{hopkins-thesis,KWB-notes}.

Recent work has revealed that low-degree polynomials appear to be a good proxy for the inherent computational tractability of many high-dimensional testing problems. For many classical problems---including planted clique, sparse PCA, community detection, tensor PCA, and more---it has been shown that polynomials of degree $D = O(\log N)$ are as powerful as the best known polynomial-time algorithms. In other words, for all of these problems, in the parameter regime where a poly-time algorithm is known we have $\|L_N^{\le O(\log N)}\| \to \infty$ whereas in the conjectured ``hard'' regime (where no poly-time algorithm is known) we have $\|L_N^{\le \omega(\log N)}\| = O(1)$ \cite{HS-bayesian,HKP-sos,hopkins-thesis,KWB-notes,DKWB-spspca}. This method even captures sharp computational thresholds such as the Kesten--Stigum bound in the stochastic block model~\cite{HS-bayesian,hopkins-thesis}. One explanation for this phenomenon is that the best known poly-time algorithms for these types of problems typically take the form of spectral methods (i.e.\ thresholding the leading eigenvalue of some matrix built from the observed data), and any such spectral method can be implemented as an $O(\log N)$-degree polynomial via power iteration (under mild conditions); see Theorem~4.4 of~\cite{KWB-notes} for the precise sense in which boundedness of $\|L_N^{\le D}\|$ imply failure of all spectral methods. Boundedness of $\|L_N^{\le D}\|$ also implies failure of \emph{statistical query algorithms}; see~\cite{ld-sq}. For larger degree $D \gg \log N$ it has also been observed that degree-$D$ polynomials are as powerful as the best known algorithms of runtime $N^{\tilde\Theta(D)}$ (which is the runtime needed to evaluate a degree-$D$ polynomial) in settings such as tensor PCA~\cite{KWB-notes} and sparse PCA~\cite{DKWB-spspca} (both of which exhibit a smooth tradeoff between signal strength and subexponential runtime). The above ideas are summarized by the following informal \emph{low-degree conjecture} based on \cite{HS-bayesian,HKP-sos,hopkins-thesis}.

\begin{conjecture}[Informal]
    \label{conj:low-deg-informal}
    Let $t: \NN \to \NN$.
    For ``natural'' high-dimensional testing problems specified by $\PP_N$ and $\QQ_N$, if $\|L_N^{\leq D(N)}\|$ remains bounded as $N \to \infty$ whenever $D(N) \le t(N) \cdot \polylog(N)$, then there exists no sequence of functions $f_N: \Omega_N \to \{\tp,\tq\}$ with $f_N$ computable in time $N^{O(t(N))}$ that strongly distinguishes $\PP_N$ and $\QQ_N$, i.e., that satisfies
    \begin{equation}
        \lim_{N \to \infty} \QQ_N\left[ f_N(Y) = \tq \right] = \lim_{N \to \infty} \PP_N\left[ f_N(Y) = \tp \right] = 1.
    \end{equation}
\end{conjecture}

\noindent This conjecture is informal because we have not attempted to specify the precise class of testing problems for which this is believed to hold. One formal variant of this conjecture is given by Conjecture~2.2.4 of~\cite{hopkins-thesis}, although it does not quite apply to the distributions that we consider in this paper. Still, we expect that Conjecture~\ref{conj:low-deg-informal} holds for our particular choice of $\PP$ and $\QQ$ (defined in Section~\ref{sec:wishart}), as these distributions are very similar in spirit to many classical settings where the conjecture is known to coincide with widely believed computational thresholds (e.g.\ spiked matrix models~\cite{BKW-sk,KWB-notes} and sparse PCA~\cite{DKWB-spspca}). We note that proving Conjecture~\ref{conj:low-deg-informal} appears to be currently out of reach, as this would imply $P\neq NP$. Furthermore, current techniques are not able to prove computational hardness of average-case problems, even under the assumption $P \neq NP$. Instead, one should think of bounds on $\|L^{\le D}\|$ as unconditional lower bounds against a restricted class of algorithms (namely low-degree polynomials). Conjecture~\ref{conj:low-deg-informal} captures the idea that low-degree algorithms are optimal among all known algorithms for a large (and growing) number of problems of this flavor.

In this paper, we give bounds on $\|L^{\le D}\|$ for our specific distributions $\PP$ and $\QQ$ defined in Section~\ref{sec:wishart}. Assuming Conjecture~\ref{conj:low-deg-informal}, this implies that runtime $N^{\tilde\Omega(s^2/M)}$ is required to distinguish them. Alternatively, the reader may choose to think of our lower bounds as \emph{unconditional} lower bounds against low-degree algorithms. This suggests that beating the runtime $N^{\tilde\Omega(s^2/M)}$ would require a drastically new algorithmic technique which would likely lead to breakthroughs in other problems as well.

\paragraph{Organization.}
The remainder of the paper is organized as follows.
In Section~\ref{sec:main-results}, we state our lower bound for average-case RIP certification based on the low-degree likelihood ratio, and show that the lower bound is optimal, matching the upper bound due to~\cite{KZ-rip}. In Section~\ref{sec:proofs-ldlr}, we give the proof for the lower bound.

\paragraph{Notation.}
Our standard asymptotic notation $O(\cdot)$, $\Omega(\cdot)$, $\Theta(\cdot)$ always pertains to the limit $N \to \infty$. We use $\polylog(N)$ to mean $(\log N)^{O(1)}$. We also use $\tilde{O}(B)$ to mean $O(B \cdot \polylog(N))$, $\tilde{\Omega}(B)$ to mean both $\Omega(B / \polylog(N))$, and $\tilde{\Theta}(B)$ to mean $\tilde{O}(B)$ and $\tilde{\Omega}(B)$. Also recall that $f(N) = o(g(N))$ (or equivalently $f(N) \ll g(N)$) means $f(N)/g(N) \to 0$ as $N \to \infty$ and $f(N) = \omega(g(N))$ (or equivalently $f(N)\gg g(N)$) means $f(N)/g(N) \to \infty$ as $N \to \infty$. We write $A \lesssim B$ to mean $A \le CB$ for an absolute constant $C$, and $A \gtrsim B$ to mean $A \ge CB$ for an absolute constant $C$. We say that an event occurs \textit{with high probability} if it occurs with probability $1-o(1)$ as $N\rightarrow\infty$.

\section{Main Results}
\label{sec:main-results}

We follow the proof sketch outlined in Proposition~\ref{pca-rip}. Adopting the framework of Proposition \ref{pca-rip}, we consider $\delta \in (0,1)$ held fixed as $N \to \infty$, with $M = M(N)$ and $s = s(N)$ scaling with $N$ such that $M \le N$, $M \to \infty$ and $s \to \infty$.

\begin{lemma}
\label{P-notrip}
Consider the setting of Proposition \ref{pca-rip}. Under $\PP$,
\begin{equation}
   \Pr\left[ \frac{1}{\sqrt{M}}A\ \text{is}\ (s,\delta)\text{-}\RIP\right]
    \le \exp\left(-\frac{\delta^2 M}{12}\right) + 2\exp\left(-\frac{(1-\delta)^2 s}{24}\right) = o(1).
\end{equation}
\end{lemma}
\noindent
We defer the proof of Lemma \ref{P-notrip} to Section~\ref{sec:proofs-ldlr}. The following well-known result tells us that the random ensemble is RIP with high probability under $\QQ$.

\begin{theorem}
[See \cite{WBP-avg-hard}, Proposition 1; also \cite{CT-decoding,baraniuk-proof,foucart-cs}]
\label{Q-rip}
In the setting of Proposition \ref{pca-rip}, under $\QQ$, \begin{equation}
   \Pr\left[ \frac{1}{\sqrt{M}}A\ \text{is not}\ (s,\delta)\text{-}\RIP\right] 
   \le 2\exp\left[s\log\left(\frac{9eN}{s}\right)-\frac{\delta^2 M}{256}\right],
\end{equation}
which is $o(1)$ provided $s \lesssim M/\log N$.
\end{theorem}

\noindent Based on Conjecture \ref{conj:low-deg-informal}, we are left to show the boundedness of the low-degree likelihood ratio between $\PP$ and $\QQ$, which gives evidence for hardness of the certification task. This is demonstrated in the following theorem.
\begin{theorem}
\label{ldlr-bdd}
In the setting of Proposition \ref{pca-rip}, suppose $M = M(N)$, $s = s(N)$, and $D = D(N)$ satisfy $s \le M$ and $D = o(s^2/M)$. Then the low-degree likelihood ratio $\|L^{\le D}\|$ of $\PP$ against $\QQ$ remains bounded as $N\rightarrow\infty$.
\end{theorem}

\noindent The proof of Theorem \ref{ldlr-bdd} is deferred to Section \ref{sec:proofs-ldlr}. The condition $s \le M$ is not restrictive because we are interested in the regime $s \lesssim M/\log N$ so that a random matrix drawn from $\QQ$ is RIP. (Our proof actually still works without the assumption $s \le M$, provided we add the additional assumption $D = o(M)$.) Theorem \ref{ldlr-bdd} tells us that for a polynomial to distinguish $\PP$ from $\QQ$, it must have degree at least $\Omega(s^2/M)$. As discussed in Section~\ref{low-deg}, this suggests that runtime $N^{\tilde\Omega(s^2/M)}$ is required to strongly distinguish $\PP$ from $\QQ$ (at least with current techniques).

Together, the results in this section verify items (1)-(3) in Proposition~\ref{pca-rip} and thus constitute evidence that RIP certification requires time $N^{\tilde{\Omega}(s^2/M)}$.

\begin{remark}[Relation to planted sparse vector problem]
Note that the orthogonal complement of the row-span of our negatively-spiked Wishart matrix is a subspace of dimension $N-M$ that (approximately) contains a sparse vector. Thus Theorem~\ref{ldlr-bdd} suggests that when $s \ge\sqrt{N} \cdot \polylog(N)$, no polynomial-time algorithm can recover a planted $s$-sparse vector in a random subspace of dimension $\Theta(N)$ in $\RR^N$. This matches the upper bound of the $l^1/l^\infty$ relaxation in~\cite{SWW-sparse-dict,DH-sparse}. Another line of work on finding a sparse vector in a subspace~\cite{BKS-sos,QSW-sparse,sos-fast} considers the setting where $s = O(N)$ and the subspace has dimension $o(N)$, but our results do not apply to this regime.
\end{remark}

For convenience to the reader, we also present below the \textit{lazy algorithm} proposed  in~\cite{KZ-rip}, which gives the matching upper bound. In the following context, for $S\subseteq [N]$ we denote by $\bP_S := \sum_{i\in S}e_i e_i^\top$ the projector that zeros out all but the entries indexed by $S$ of a vector. 
\begin{proposition}
\label{prop:rip-equiv}
    For a matrix $X\in \RR^{M\times N}$ and $s\in \NN_+$, for any $\delta\in (0,1)$, $X$ satisfies the $(s,\delta)$-$\RIP$ if and only if 
    \[ B_s(X) \colonequals \max_{\substack{S \subseteq [N] \\ |S| = s}} \|\bP_S(X^{\top} X - I_N)\bP_S\| \le\delta.\]
\end{proposition}
\noindent
Proposition~\ref{prop:rip-equiv} follows directly from the definition of RIP. A na\"ive way to certify RIP is to compute $B_s(X)$ by enumerating all $S \subseteq [N]$ with $|S| = s$, which would take time $N^{\Theta(s)}$. Instead, the following algorithm calculates $B_r(X)$ for $r = \tilde{\Theta}(s^2/M)$, thereby reducing the total runtime to $N^{\tilde{\Theta}(s^2/M)}$.

\noindent\begin{minipage}{\linewidth}
\begin{algorithm}[H]
\caption{~\cite{KZ-rip} \textit{Lazy algorithm:} Certification for the $(s,\delta)$-$\RIP$}
\label{algo:rip-cert}
\begin{algorithmic}[1]
\REQUIRE Input matrix $X\in \RR^{M\times N}$ with unit column vectors; parameters $r, s\in [N]$ with $1 < r \leq s$, and a number $\delta\in (0,1)$.
\STATE Compute $B_r(X)$.
\IF{$\frac{s-1}{r-1} B_r(X) \le \delta$}
\RETURN ``yes"
\ELSE
\RETURN ``no"
\ENDIF
\end{algorithmic}
\end{algorithm}
\end{minipage}
\vspace{5pt}

\noindent
It is proved in~\cite{KZ-rip} that, with a choice $r = \tilde{\Theta}\left(s^2/(\delta^2 M)\right)$, Algorithm~\ref{algo:rip-cert} certifies the $(s,\delta)$-$\RIP$ for $X = \frac{1}{\sqrt{M}}A$ with high probability, where $A$ is an $M\times N$ matrix with i.i.d.\ symmetric ($\pm 1$) Bernoulli random variables as its entries. We remark that with minor changes the same procedure works for $A$ with i.i.d.\ standard Gaussian random variables, and possibly also for many other ensembles. Based on our lower bound and the effectiveness of Algorithm~\ref{algo:rip-cert}, we conclude that the task of $(s,\delta)$-$\RIP$ certification requires time precisely $N^{\tilde{\Theta}(s^2/M)}$. However, our lower bound (Theorem~\ref{ldlr-bdd}) does not seem to easily generalize to other ensembles, since we have used gaussianity in a crucial way.

\begin{remark}[Relation to sparse PCA]
We remark that the RIP certification problem bears resemblance to sparse PCA, which is the positively-spiked ($\beta > 0$) case of the spiked Wishart model with a sparse spike prior. In~\cite{DKWB-spspca}, the authors investigated the precise runtime required to solve sparse PCA. This includes an algorithm that improves over the runtime of na\"ive exhaustive search by enumerating subsets of a particular size smaller than the true sparsity (similar to the lazy algorithm above and to the sparse PCA algorithm of~\cite{anytime-pca}), as well as a matching lower bound based on the low-degree likelihood ratio.
\end{remark}

\section{Proof for the Low-Degree Likelihood Ratio Bound}
\label{sec:proofs-ldlr}

In this section we prove Lemma \ref{P-notrip} and Theorem \ref{ldlr-bdd}. We start with introducing the following two Chernoff-type bounds for Bernoulli and $\chi^2$ sums.
 
\begin{lemma}
\label{chern-ber}
Suppose $x$ is taken from the sparse Rademacher prior $\sX_{\rho}^N$ per Definition \ref{sps-rad}. For any $\mu\in (0,1]$, we have 
\begin{equation}
\label{ber-conc}
    \Pr\left[\|x\|^2 > 1+\mu \right] \le \exp\left(-\frac{\mu^2\rho N}{3}\right), \quad \Pr\left[\|x\|^2 < 1-\mu \right] \le \exp\left(-\frac{\mu^2\rho N}{2}\right),
\end{equation}
and therefore
\begin{equation}
\label{ber-conc2}
    \Pr\left[1-\mu \le \|x\|^2 \le  1+\mu \right] \ge 1-2\exp\left(-\frac{\mu^2\rho N}{3}\right).
\end{equation}
\end{lemma}
\begin{proof}
Note that 
\[\Pr\left[\|x\|^2 > 1+\mu \right] = \Pr\left[\|x\|_0 > (1+\mu)\rho N\right],\quad \Pr\left[\|x\|^2 < 1-\mu \right] = \Pr\left[\|x\|_0 < (1-\mu)\rho N\right] \]
where $\|x\|_0$ is the sum of $N$ independent Bernoulli$(\rho)$ random variables, and $\EE\|x\|_0 = \rho N$. Therefore (\ref{ber-conc}) and (\ref{ber-conc2}) follow from the multiplicative Chernoff bound in  \cite{AV-chernoff}.
\end{proof}

\begin{lemma}[Chernoff bound for $\chi^2$ distribution]
\label{chern-chi2}
For all $\delta\in (0,1)$,
\begin{equation}
    \Pr\left[\frac{1}{M}\chi_M^2 \ge 1+\delta\right] \le \exp\left(-\frac{\delta^2 M}{12}\right).
\end{equation}
\end{lemma}
\begin{proof}
By the classical Chernoff bound (see, for example \cite{LM-chernoff}),
\[\frac{1}{M}\Pr\left[\chi_M^2 \ge (1+\delta)M\right] \le \frac{1}{2}(-\delta +\log(1+\delta)).\]
Now the lemma follows immediately from the observation that for $\delta\in (0,1)$,
\[\frac{1}{2}(-\delta+\log(1+\delta)) \le -\frac{\delta^2}{12}.\]
\end{proof}

\begin{proof}[Proof of Lemma \ref{P-notrip}]
Note that
\[ \left\|\frac{1}{\sqrt{M}}Ax\right\|^2 = \frac{1}{M}\sum_{i = 1}^M (u_i^\top x)^2 \sim \frac{1}{M}\left(\|x\|^2-(1-\epsilon)\|x\|^4\right)\chi_M^2. \]
For $x$ taken from $\sX_N^\rho$ satisfying $1-\epsilon \le \|x\|^2\le 1+\epsilon$, since $-(1-\epsilon)\|x\|^2 \ge -(1-\epsilon^2) > -1$, under $\PP$ each row of the observation $A$ is taken from $\sN(0,I-(1-\epsilon)xx^\top)$. Furthermore, we deduce from $\|x\|^2\le 1+\epsilon$ that $\|x\|_0 \le (1+\epsilon)\rho_N N \le s$. Now Lemma \ref{chern-chi2} gives
\begin{align*}
    \Pr\left[\left\|\frac{1}{\sqrt{M}}Ax\right\|^2 \ge (1-\delta)\|x\|^2 \right] 
    &= \Pr\left[\frac{1}{M}\chi_M^2 \ge \frac{1-\delta}{1-(1-\epsilon)\|x\|^2} \right]\\
    &\le \Pr\left[\frac{1}{M}\chi_M^2 \ge 1+\delta \right]\\
    &\le \exp\left(-\frac{\delta^2 M}{12}\right).
\end{align*}
The first inequality used the fact that
\[\frac{1-\delta}{1-(1-\epsilon)\|x\|^2} \ge \frac{1-\delta}{1-(1-\epsilon)^2} 
\ge \frac{1-\delta}{2\epsilon} = 1+\delta.\] 
Therefore we know from Lemma \ref{chern-ber} that
\begin{align*}
    &\Pr\left[\frac{1}{\sqrt{M}}A\ \text{is}\ (s,\delta)\text{-}\RIP\right]\\
    \le & \Pr\left[1-\epsilon\le \|x\|^2 \le 1+\epsilon,\ \left\|\frac{1}{\sqrt{M}}Ax\right\|^2 \ge (1-\delta)\|x\|^2 \right] +\Pr\left[\|x\|^2 < 1-\epsilon\right] + \Pr\left[\|x\|^2 > 1+\epsilon\right]\\
    \le & \exp\left(-\frac{\delta^2 M}{12}\right) + 2\exp\left(-\frac{(1-\delta)^2 s}{24}\right)
\end{align*}
the rightmost sum being $o(1)$ given $M\rightarrow\infty$ and $s\rightarrow\infty$ as $N\rightarrow\infty$.
\end{proof}

\begin{proof}[Proof of Theorem \ref{ldlr-bdd}]
Let $L_{N,M,\beta,\sX}^{\le D}$ denote the degree-$D$ likelihood ratio for the spiked Wishart model (see Definition \ref{def-wish}) with parameters $N,M,\beta$ and spike prior $\sX$. \cite{BKW-sk} gives the formula
\begin{align}\label{ldlr_wsh}
\|L_{N,M,\beta,\sX}^{\le D}\|_{2}^2 &= \Ex_{v^{(1)},v^{(2)}\sim \sX_N}\left[\varphi_{M,\lfloor D/2 \rfloor}\left(\frac{\beta^2 \la v^{(1)},v^{(2)} \ra^2}{4}\right)\right]\nonumber\\
&= \Ex_{v^{(1)},v^{(2)}\sim \sX_N}\sum_{d = 0}^{\lfloor D/2 \rfloor}\left(\sum_{\substack{d_1,\dots,d_M\\ \sum d_i = d}}\prod_{i = 1}^M\binom{2d_i}{d_i}\right)\left(\frac{\beta^2\la v^{(1)},v^{(2)}\ra^2}{4}\right)^d,
\end{align}
where $v^{(1)},v^{(2)}$ are drawn independently from $\sX_N$. Here $\varphi_{N,k}(x)$ is the Taylor series of $\varphi_N$ around $x = 0$ truncated to degree $k$, i.e.
\begin{align*}
    \varphi_M(x) &\colonequals (1-4x)^{-M/2} \\
    \varphi_{M,k}(x) &\colonequals \sum_{d = 0}^k x^d \frac{1}{d!}\prod_{a = 0}^{d - 1}\left(2M + 4a\right)
\end{align*}
by the the generalized binomial theorem, where the coefficient of $x^d$ may be written in terms of a generalized binomial coefficient as $(-4)^d \binom{-M / 2}{d}$.\footnote{We note that this expansion of the power series, which allows a simplified analysis, was not noticed in \cite{BKW-sk}.}
Note that under the setting of Problem \ref{pca-rip}, we are essentially dealing with $\sX_N$ a truncated version of the sparse Rademacher prior $\sX_N^{\rho_N}$: $x\sim \sX_N$ is taken as following.
\begin{itemize}
    \item[(1)] Draw $x\sim \sX_N^{\rho_N}$.
    \item[(2)] If $-(1-\epsilon)\|x\|^2 < -1$, then set $x = 0$.
\end{itemize}
Therefore for any non-negative integer $d$ it holds that
\begin{equation}
    \Ex_{v^{(1)},v^{(2)}\sim \sX_N} \la v^{(1)},v^{(2)}\ra^{2d} \le \Ex_{v^{(1)},v^{(2)}\sim \sX_N^{\rho_N}} \la v^{(1)},v^{(2)}\ra^{2d}.
\end{equation}
For any $d \leq D$, the power series coefficients above are bounded by
\begin{equation}
    0 \leq \frac{1}{d!}\prod_{a = 0}^{d - 1}\left(2M + 4a\right) \leq \frac{(2M + 4D)^d}{d!}
\end{equation}
For independent choices of $v^{(1)}, v^{(2)} \sim \sX_N^{\rho_N}$, denote $S^{(1)}$ and $S^{(2)}$ their respective support. Observe that
\[\la v^{(1)},v^{(2)}\ra|S^{(1)}, S^{(2)} \stackrel{d}{=} \frac{1}{\rho N}\sum_{i\in S^{(1)}\cap S^{(2)}}R_i\]
where $R_i$ are i.i.d.\ Rademacher random variables. Following Section 4.2 of~\cite{LWB-sparse}, we get
\begin{align}
    \Ex_{v^{(1)},v^{(2)}\sim \sX_N^{\rho_N}} \la v^{(1)},v^{(2)}\ra^{2d}
    &= \frac{1}{(\rho N)^{2d}}\Ex_{S^{(1)}, S^{(2)}} \left[\Ex \left(\sum_{i\in S^{(1)}\cap S^{(2)}} R_i\right)^{2d} \, \bigg| \,  S^{(1)}, S^{(2)} \right]\nonumber\\
    &\le \frac{1}{(\rho N)^{2d}} (2d-1)!! \Ex_{S^{(1)}, S^{(2)}}|S^{(1)}\cap S^{(2)}|^d\nonumber \\
    &= \frac{1}{(\rho N)^{2d}} (2d-1)!!\cdot \Ex \left(\sum_{i = 1}^N B_i\right)^d \label{eq:bern-for-cap}\\
    &\le \frac{1}{(\rho N)^{2d}} (2d-1)!!
    \left[\rho^2 N + 2^{1/d} \left(d^{d/2}4^{d-1}(\rho^2 N)^{d/2} + \left(\frac{4}{3}d\right)^d\right)^{1/d}\right]^d\nonumber\\
    & \le \frac{1}{(\rho N)^{2d}} (2d-1)!!
    \left(\rho^2 N + 4\rho \sqrt{dN} + 3d\right)^d \nonumber
\end{align}
where, in~\eqref{eq:bern-for-cap}, $B_i = \one\{i\in S^{(1)}\cap S^{(2)}\}$ are i.i.d.\ Bernoulli random variables with success probability $\rho^2$. Overall, we obtain that
\begin{align}
    \|L_{N,M,\beta,\sX}^{\le D}\|_{2}^2 
    &\lesssim \sum_{d = 0}^{\lfloor D/2 \rfloor} \left(\frac{\beta^2 (M + 2D)}{
    \rho^2 N^2}\right)^d \left(\rho^2 N + 4\rho \sqrt{dN} + 3d\right)^d \nonumber\\
    &\le \sum_{d = 0}^{\lfloor D/2 \rfloor} \left(\frac{\beta^2 (M + 2D)}{
    \rho^2 N^2}\right)^d \left(\rho^2 N + 4\rho \sqrt{DN/2} + \frac{3D}{2}\right)^d \nonumber\\
    & \le \sum_{d = 0}^{\lfloor D/2 \rfloor} \left(\beta^2 \left(\frac{M + 2D}{N}+\frac{4\sqrt{2D}(M + 2D)}{s\sqrt{N}} + \frac{6D(M + 2D)}{s^2}\right) \right)^d\nonumber\\
    & = \sum_{d = 0}^{\lfloor D/2 \rfloor} \left((1+o(1))(1-\epsilon)^2 \left(\frac{M}{N}+4\sqrt{2} \sqrt{\frac{DM}{s^2}} \sqrt{\frac{M}{N}} + 6 \frac{DM}{s^2}\right) \right)^d\nonumber\\
    & = O(1) \nonumber
\end{align}
where we have used $\beta = -(1-\epsilon)$ and $M \le N$ along with the assumptions from Theorem~\ref{ldlr-bdd}: $s \le M$ and $D = o(s^2/M)$, which together imply $D = o(M)$.
\end{proof}

\section*{Acknowledgments}

We thank Cheng Mao for helpful comments on an earlier version.

\addcontentsline{toc}{section}{References}
\bibliographystyle{alpha}
\bibliography{main}

\end{document}